\documentclass[a4paper,11pt]{amsart}

\usepackage[utf8]{inputenc}
\usepackage[toc]{appendix}

\usepackage[english]{babel}

\usepackage{graphicx}
\usepackage{subfigure}
\usepackage[justification=centering]{caption}

\usepackage{amsmath,a4wide}
\usepackage{amsfonts}
\usepackage{amssymb}
\usepackage{amsthm}
\usepackage{mathtools}
\usepackage{amsaddr}

\usepackage{enumerate}
\usepackage{enumitem}

\usepackage{hyperref}
\usepackage{url}

\numberwithin{equation}{section}
\mathtoolsset{showonlyrefs}
\providecommand{\norm}[1]{\big\lVert#1\big\rVert}
\newcommand{\R}{\mathbb{R}}

\newcommand{\Rd}[1][d]{\mathbb{R}^{#1}}
\newcommand{\Z}{\mathbb{Z}}
\newcommand{\N}{\mathbb{N}}
\newcommand{\Lt}[1][d]{L^2(\R^{#1})}
\newcommand{\A}{\mathcal{A}}
\newcommand{\G}{\mathcal{G}}
\newcommand{\F}{\mathcal{F}}

\newcommand{\V}{\mathcal{V}}
\newcommand{\W}{\mathcal{W}}

\newcommand{\Lp}[1]{L^#1}
\newcommand{\lp}[1]{\ell^#1}

\renewcommand{\l}{\lambda}
\renewcommand{\L}{\Lambda}
\newcommand{\vol}{\textnormal{vol}}

\newcommand{\J}{
	\left(
		\begin{array}{rc}
		0 & I\\
		-I & 0
		\end{array}
	\right)
}

\DeclareMathOperator*{\essinf}{ess\,inf}
\DeclareMathOperator*{\esssup}{ess\,sup}

\theoremstyle{plain}
\newtheorem{theorem}{Theorem}[section]

\theoremstyle{plain}

\theoremstyle{plain}
\newtheorem{lemma}[theorem]{Lemma}

\theoremstyle{plain}
\newtheorem{proposition}[theorem]{Proposition}

\theoremstyle{definition}
\newtheorem{definition}[theorem]{Definition}

\theoremstyle{remark}

\theoremstyle{remark}

\theoremstyle{definition}

\begin{document}
\title[A Short Note on the Frame Set of Odd Functions]{A Short Note on the Frame Set of Odd Functions}
\author[Markus Faulhuber]{Markus Faulhuber}
\address{Analysis Group, Department of Mathematical Sciences, NTNU Trondheim\\ Sentralbygg 2, Gløshaugen, Trondheim, Norway}
%\address{NuHAG, Faculty of Mathematics, University of Vienna, Oskar-Morgenstern-Platz 1, 1090 Vienna, Austria}
\email{markus.faulhuber@ntnu.no}
\thanks{The author was supported by the Erwin--Schrödinger program of the Austrian Science Fund (FWF): J4100-N32. The computational results presented have been achieved (in part) using the Vienna Scientific Cluster (VSC). The author wishes to thank Franz Luef for many fruitful discussions on the topic. The author thanks the anonymous referee for beneficial feedback on the first version of the manuscript.}

\begin{abstract}
	In this work we derive a simple argument which shows that Gabor systems consisting of odd functions of $d$ variables and symplectic lattices of density $2^d$ cannot constitute a Gabor frame. In the 1--dimensional, separable case, this is a special case of a result proved by Lyubarskii and Nes, however, we use a different approach in this work exploiting the algebraic relation between the ambiguity function and the Wigner distribution as well as their relation given by the (symplectic) Fourier transform. Also, we do not need the assumption that the lattice is separable and, hence, new restrictions are added to the full frame set of odd functions.
\end{abstract}


\maketitle

\section{Introduction and Main Result}\label{sec_Intro}
In this short note we show that the full frame set of any odd function of $d$ variables in Feichtinger's algebra cannot contain symplectic lattices of density $2^d$. In the 1--dimensional, separable case, this is a special case of a more general result derived by Lyubarskii and Nes \cite{LyubarskiiNes_Rational_2013} who could show that no odd window function $g \in S_0(\R)$ can produce a separable Gabor frame of redundancy $\frac{n+1}{n}$, $n \in \N$ by studying the vector--valued Zak transform and Zebulski--Zeevi matrices. For an alternative proof of this result see the survey article by Gröchenig \cite{Gro14}.

However, our arguments are somewhat simpler and hold for symplectic lattices in arbitrary dimension $d$, which makes up for the drawback that we do not derive more general results. The key argument is that the Wigner distribution is the symplectic Fourier transform of the ambiguity function and that they also fulfill a simple algebraic relation. Moreover, our arguments show that, after a proper scaling, the cross Wigner distribution of any function in Feichtinger's algebra and any even function in Feichtinger's algebra is an eigenfunction of the symplectic Fourier transform with eigenvalue 1 and the pairing with any odd function in Feichtinger's algebra is an eigenfunction with eigenvalue -1.
%In addition, this note underlines that lattices of density $2^d$ take a special role for ($d$--dimensional) Gabor frames.
%With basically the same arguments we show the impossibility of generating a Gabor frame with a window from Feichtinger's algebra at critical density.
%We note that all proofs can be easily adjusted to windows in $S_0(\Rd)$ with $d > 1$ and symplectic lattices of density $2^d$.

This work concerns the fine structure of Gabor frames as described in \cite{Gro14}, i.e., relations between the properties of a fixed window and its frame set. For a (window) function $g \in \Lt$ and an index set $\L \subset \Rd[2d]$, we denote the resulting Gabor system by $\G(g,\L)$. The (full) frame set of the window $g$ is given by
\begin{equation}
	\mathfrak{F}_{full}(g) = \{ \L \subset \Rd[2d], \L \text{ a lattice} \mid \G(g,\L) \text{ is a frame} \}.
\end{equation}
Inspired by the work of Lemvig \cite{Lemvig_Hermite_2016}, the original intention of this short note was to show up simple restrictions for the full frame set of odd (1--dimensional) Hermite functions by showing that certain sums vanish, but the restriction to this very special class of functions turned out to be unnecessary. Unfortunately, we do not get any new insights into the frame set of even (Hermite) functions. Among other counterexamples, Lemvig showed that the square lattice of density 2 does not generate a Gabor frame for the second Hermite function (the Gaussian being indexed as 0--th Hermite function), which was the first known obstruction to the frame set of the second Hermite function. Numerical inspections suggest that, for the second Hermite function, among all separable lattices of density 2 the square lattice is the only lattice which does not yield a Gabor frame, in particular, in case of the square lattice the lower frame bound is zero and it yields the global minimum of the lower frame bound seen as a function of the lattice parameters. This example stands in sharp contrast to the results given in \cite{FaulhuberSteinerberger_Theta_2017}, where it is shown that under the same assumptions, but using the Gaussian instead of the second Hermite function, the square lattice gives the global maximum of the lower frame bound seen as a function of the lattice parameters. The common theme, however, is that in both cases the highest possible symmetry of the lattice leads to extremal frame bounds. It was proven in \cite{Faulhuber_Hexagonal_2018} that, for a Gabor frame of even redundancy with standard Gaussian window, the hexagonal lattice yields the smallest upper frame bound among all lattices. We conjecture that the hexagonal lattice should also give the largest lower frame bound in this case. So, we pose the following question: For the second Hermite function, does the Gabor system generated by  the hexagonal lattice of density 2 have a positive lower frame bound? The results by Lemvig tempt us to think that this might not the be case, but numerical inspections say that we actually have a Gabor frame with approximate lower frame bound $0.29 \dots$ .

Our main result, however, concerns odd windows in Feichtinger's algebra which we denote by $S_0(\Rd)$ (another common notation is $M^1(\Rd)$).

\begin{theorem}[Main Result]\label{thm_main}
	Let $g \in S_0(\Rd)$ be an odd function, i.e., $g(t) = -g(-t)$ and let $\L \subset \Rd[2d]$ be a symplectic lattice in the time--frequency plane. If $\vol(\L) = 2^{-d}$ then $\G(g, \L)$ cannot be a Gabor frame, or, in shorter notation:
	\begin{equation}
		\text{If } g \in S_0(\Rd), \, g(t) = -g(-t) \text{ and } \vol(\L) = 2^{-d}, \, \L \text{ symplectic} \quad \Longrightarrow \quad \L \notin \mathfrak{F}_{full}(g).
	\end{equation}
\end{theorem}
Theorem \ref{thm_main} particularly implies that for $d=1$ no lattice of density 2 can be contained in the frame set of an odd function from Feichtinger's algebra.

This work is structured as follows:
\begin{itemize}
	\item In Section \ref{sec_TFA} we recall the basic properties of Gabor frames for the Hilbert space $\Lt$. After that, we introduce quadratic representations of a function $f \in \Lt$ with respect to a (fixed) window $g \in \Lt$, namely the short--time Fourier transform, the ambiguity function and the Wigner distribution. We show their algebraic relations as well as their relation under the symplectic Fourier transform and introduce the symplectic version of Poisson's summation formula. Also, we will see that Feichtinger's algebra is a convenient setting for our purposes.
	
	\medskip	
	
	\item In Section \ref{sec_sharp_bounds} we show how sharp frame bounds can be calculated, using the results established by Janssen in the 1990s. These results finally lead to the proof of Theorem \ref{thm_main}.
	
	%\medskip
	
	%\item In Section \ref{sec_critical} we prove, by using the same simple arguments as in the previous sections, that for an even window in Feichtinger's algebra no Gabor frames can exist at critical denisty.
\end{itemize}

\section{Gabor Frames and Time--Frequency Analysis in a nutshell}\label{sec_TFA}
We consider Gabor frames for the Hilbert space of square integrable functions in $d$--dimensional Euclidean space $\Lt$. Concerning the notation we follow mainly the textbook of Gröchenig \cite{Gro01}. A more recent introduction to the topic is the $2^{nd}$ edition of Christensen's textbook \cite{Christensen_2016}.

As our functions will be defined pointwise and at least continuous in the remainder of this work the following notation for the inner product in $\Lt$ is justified;
\begin{equation}
	\langle f, g \rangle = \int_{\Rd} f(t) \overline{g(t)} \, dt.
\end{equation}
For two vectors $t$ and $t'$ in $\Rd$ we denote the Euclidean scalar product by $t \cdot t'$.

The key elements in time-frequency analysis are the translation operator $T_x$ (time shift) and the modulation operator $M_\omega$ (frequency shift) which are defined as
\begin{equation}
	T_x f(t) = f(t-x) \qquad \text{ and } \qquad M_\omega f(t) = e^{2 \pi i \omega \cdot t} f(t).
\end{equation}
For a function in the Schwartz space $\mathcal{S}(\Rd)$ we define the Fourier transform by
\begin{equation}
	\F f(\omega) = \int_\R f(t) e^{-2 \pi i \omega \cdot t} \, dt,
\end{equation}
which extends to a unitary operator on $\Lt$ by the usual density argument. The Fourier transform has the well--known properties of interchanging translation and modulation, i.e.,
\begin{equation}
	\F (T_xf) = M_{-x} \, \F f \qquad \text{ and } \qquad \F (M_\omega f) = T_\omega \, \F f.
\end{equation}
The translation (time shift) and modulation (freuqency shift) operator do not commute in general, but they fulfill the following commutation relation
\begin{equation}\label{eq_comm_rel}
	M_\omega T_x = e^{2 \pi i \omega \cdot x} T_x M_\omega.
\end{equation}
Hence, the combination of the two operators is called a time--frequency shift and usually denoted by
\begin{equation}
	\pi (\l) = M_\omega T_x, \qquad \l = (x,\omega) \in \R^2,
\end{equation}
where $\l$ is a point in the time--frequency plane or phase space. The composition of two time--frequency shifts is given by
\begin{equation}
	\pi(\l)\pi(\l') = e^{-2 \pi i x \cdot \omega'} \pi(\l + \l').
\end{equation}

A Gabor system is a collection of time--frequency shifted copies of a so--called window function $g \in \Lt[]$ with respect to an index set $\L \subset \R^2$ and it is denoted by
\begin{equation}
	\G(g, \L) = \{\pi(\l) g \mid \l \in \L \}.
\end{equation}
Throughout this work, $\L$ will be a lattice, i.e., a discrete subgroup of $\R^{2d}$. A lattice can be represented by an invertible matrix $M \in GL(2d,\R)$ and is then given by $\L = M \Z^{2d}$. The matrix $M$ is not unique since we can choose from countably many possible bases for $\Z^{2d}$. For example, if $d=1$, then any matrix $\mathcal{B}$ with integer entries and determinant 1, i.e., $\mathcal{B} \in SL(2,\Z)$, satisfies $\mathcal{B} \Z^2 = \Z^2$. Although the representing matrix is not unique its determinant is. We define the volume of a lattice $\L = M \Z^{2d}$ by
\begin{equation}
	\vol(\L) = |\det(M)|.
\end{equation}
The density of a lattice is given by the reciprocal of the volume, i.e., $\delta(\L) = \vol(\L)^{-1}$. Usually, a lattice is called separable if it can be written as $\alpha \Z^d \times \beta \Z^d$, $\alpha, \beta \in \R_+$. Alternative definitions of a separable lattice are that the generating matrix is a diagonal matrix, or, in the most general case, that the lattice separates as $M_1 \Z^d \times M_2 \Z^d$ with $M_1$, $M_2 \in GL(d,\R)$. For $d=1$ all definitions coincide.

%The treatise of Gabor systems with separable lattices of the form $\alpha \Z^d \times \beta \Z^d$ is quite common in literature and also well-described in standard textbooks (e.g.~\cite{Christensen_2016,Gro01}).

A Gabor system $\G(g,\L)$ is called a Gabor frame if and only if the frame inequality is fulfilled, i.e.,
\begin{equation}\label{eq_frame}
	A \norm{f}^2 \leq \sum_{\l \in \L} \left| \langle f , \pi(\l) g \rangle \right|^2 \leq B \norm{f}^2, \quad \forall f \in \Lt,
\end{equation}
with positive constants $0 < A \leq B < \infty$ called frame bounds. In general, a Gabor frame is a redundant system and the redundancy of a Gabor system is given by the density of the underlying lattice. If all elements of the Gabor system $\G(g,\L)$ have unit norm, the redundancy also reflects itself in the frame bounds. We note that in the case of an orthonormal basis we have $ A = B = 1$.

\subsection{Symmetric Time--Frequency Shifts}\label{subsec_symmetric}
It will be advantageous to consider symmetric time--frequency shifts instead of usual time--frequency shifts. The symmetric time--frequency shift is given by
\begin{equation}\label{eq_TF_shift_symmetric}
	\rho(\l) = M_{\tfrac{\omega}{2}} T_x  M_{\tfrac{\omega}{2}} = T_{\tfrac{x}{2}} M_\omega  T_{\tfrac{x}{2}} = e^{-\pi i x \cdot \omega} \, \pi(\l).
\end{equation}
We note that
\begin{equation}
	\rho(\l) \rho(\l') = e^{-\pi i (x \cdot \omega' - x' \cdot \omega)} \rho(\l+\l').
\end{equation}
The Gabor system under consideration is then
\begin{equation}
	\widetilde{\G}(g, \L) = \{ \rho(\l) g \mid \l \in \L \}.
\end{equation}
This system is a frame if and only if there exist positive constants $0 < A \leq B < \infty$ such that
\begin{equation}\label{eq_frame_symmetric}
	A \norm{f}^2 \leq \sum_{\l \in \L} \left| \langle f , \rho(\l) g \rangle \right|^2 \leq B \norm{f}^2, \quad \forall f \in \Lt,
\end{equation}
It follows from \eqref{eq_TF_shift_symmetric} that the optimal constants $A,B$ in equations \eqref{eq_frame} and \eqref{eq_frame_symmetric} are the same. In particular, $\G(g, \L)$ is a frame if and only if $\widetilde{\G}(g, \L)$ is a frame. In the rest of this work we will work with the Gabor system $\widetilde{\G}(g, \L)$ as the phase factors are easier to handle in this case.

\subsection{Phase--Space Methods}
The short--time Fourier transform (STFT) and the ambiguity function are often used to measure time frequency concentration. They are defined in similar ways and, in fact, they only differ by a phase factor, i.e., a complex exponential of modulus 1. We will now introduce the necessary tools to prove Theorem \ref{thm_main}. For more details we refer to the textbooks of Folland \cite{Fol89}, de Gosson \cite{Gos11, Gosson_Wigner_2017} or Gröchenig \cite{Gro01}.
\begin{definition}[STFT]
	For $f \in \Lt$, the short--time Fourier transform with respect to the window $g \in \Lt[]$ is defined as
	\begin{equation}
		\V_g f(x, \omega) = \int_{\Rd} f(t) \overline{g(t-x)} e^{-2 \pi i \omega \cdot t} \, dt = \langle f, \pi(\l) g \rangle, \qquad \l = (x,\omega) \in \R^{2d}.
	\end{equation}
\end{definition}
Before we continue, we introduce the function space which will be most suitable for our intentions, namely Feichtinger's algebra $S_0(\Rd)$, introduced by Feichtinger in the early 1980s \cite{Fei81}. There are several equivalent definitions of $S_0(\Rd)$ and we prefer to use the following definition.
\begin{definition}[Feichtinger's Algebra]
	Feichtinger's algebra $S_0(\Rd)$ consists of all elements $g \in \Lt$ such that
	\begin{equation}
		\norm{\V_g g}_{L^1(\R^{2d})} = \iint_{\R^{2d}} \left| V_g g(\l) \right| \, d\l < \infty, \qquad \l = (x,\omega) \in \R^{2d} .
	\end{equation}
\end{definition}
We note the following properties of $S_0(\Rd)$. It is a Banach space, invariant under the Fourier transform and time--frequency shifts. It contains the Schwartz space $\mathcal{S}(\Rd)$ and it is dense in $\Lp{p}(\Rd)$, $p \in [1,\infty [$. It is for these properties that it is a quite popular function space in time--frequency analysis and the literature on the subject is huge. For more details on $S_0$ we refer to the survey by Jakobsen \cite{Jakobsen_S0_2018} and the references therein.

We turn to another time--frequency representation, which is defined similarly to the STFT.
\begin{definition}[Ambiguity Function]
	For $f, \, g \in \Lt$, the (cross) ambiguity function is defined as
	\begin{align}
		\A_g f(x, \omega) & = \int_{\Rd} f(t+\tfrac{x}{2}) \overline{g(t-\tfrac{x}{2})} e^{-2 \pi i \omega \cdot t} \, dt\\
		& = \langle \pi(-\tfrac{\l}{2}) f, \pi(\tfrac{\l}{2}) g \rangle = \langle f, \rho(\l) g \rangle, \qquad \l = (x,\omega) \in \R^{2d}.
	\end{align}
\end{definition}
Both, $\V_g f$ and $\A_g f$ are uniformly continuous on $\R^{2d}$. Due to relation \eqref{eq_TF_shift_symmetric}, which is a consequence of the commutation relation \eqref{eq_comm_rel}, we have that
\begin{equation}
	\A_g f(x,\omega) = e^{\pi i \omega \cdot x} \V_g f(x,\omega).
\end{equation}
In particular this means that $|\V_gf| \equiv |A_gf|$. We will now introduce a quadratic representation of a function $f \in \Lt$ which is usually used in quantum mechanics, the Wigner distribution.
\begin{definition}[Wigner Distribution]
	For $f, \, g \in \Lt$, the (cross) Wigner distribution is defined as
	\begin{equation}
		\W_g f(x, \omega) = \int_\R f(x+\tfrac{t}{2}) \overline{g(x-\tfrac{t}{2})} e^{-2 \pi i \omega \cdot t} \, dt, \qquad x,\omega \in \Rd.
	\end{equation}
\end{definition}
For the rest of this work, we will drop the index in all of the above definitions if $f = g$. The Wigner distribution is related to the ambiguity function (and, hence, in a similar way to the STFT) by the symplectic Fourier transform. In order to define the symplectic Fourier transform, we first equip our phase space with a symplectic structure. In what follows the vectors $\l = (x,\omega)$ and $\l' = (x', \omega')$ in $\R^{2d}$ are always seen as column vectors and the scalar product of two vectors in the phase space is again denoted by $\l \cdot \l'$. We define the symplectic form
\begin{equation}
	\sigma(\l,\l') = x \cdot \omega' - \omega \cdot x' = \l \cdot J \l' = \l^T J \l,
\end{equation}
where $J = \J$ is the standard symplectic matrix and $I$ the $d \times d$ identity matrix. A matrix $S$ is called symplectic if and only if it preserves the symplectic form, i.e.,
\begin{equation}
	\sigma(S \l, S \l') = \sigma(\l, \l'),
\end{equation}
or, equivalently,
\begin{equation}
	S^T J S = J.
\end{equation}

As mentioned, it will turn out to be convenient to use a slightly different version of the Fourier transform in phase space, the symplectic Fourier transform.
\begin{definition}[Symplectic Fourier Transform]
	For $F \in \mathcal{S}(\R^{2d})$ the symplectic Fourier transform is given by
	\begin{equation}
		\F_\sigma F (x, \omega) = \iint_{\R^2} F(\l') e^{-2 \pi i \, \sigma(\l,\l')} \, d\l', \quad \l = (x,\omega), \, \l' = (x', \omega') \in \R^{2d}.
	\end{equation}
\end{definition}
Of course, the symplectic Fourier transform extends to all of $\Lt[2d]$ by the usual density argument (just as the Fourier transform). A tool which is heavily exploited in time-frequency analysis is the Poisson summation formula which we will use for $2d$--dimensional lattices. The technical details for the Poisson summation formula to hold pointwise have been worked by Gröchenig in \cite{Gro_Poisson_1996}. Since our functions under consideration are in $S_0(\Rd)$, their Wigner distributions  as well as their ambiguity functions will be elements of Feichtinger's algebra in phase space, i.e., elements of $S_0(\R^{2d})$ (see \cite[chap.~5]{Jakobsen_S0_2018}). This assumption is sufficient for Poisson's summation formula to hold pointwise.
\begin{proposition}[Poisson Summation Formula]
	For $F \in S_0(\R^{2d})$ and a lattice $\L = M \Z^{2d}$ with dual lattice $\L^\bot = M^{-T} \Z^{2d}$ we have
	\begin{equation}
		\sum_{\l \in \L} F(\l+z) = \vol(\L)^{-1} \sum_{\l^\bot \in \L^\bot} \F F(\l^\bot) e^{2 \pi i \l^\bot \cdot z}, \quad \l, \l^\bot, z \in \R^{2d}.
	\end{equation}
\end{proposition}
Instead of using the $2d$--dimensional Fourier transform we can adjust this result by using the symplectic Fourier transform and the adjoint lattice instead of the dual lattice. The adjoint of a lattice $\L = M \Z^{2d}$ is given by $\L^\circ = J M^{-T} \Z^{2d}$. Under the assumptions of Poisson's summation formula we get
\begin{equation}
	\sum_{\l \in \L} F(\l+z) = \vol(\L)^{-1} \sum_{\l^\circ \in \L^\circ} \F_\sigma F(\l^\circ) e^{2 \pi i \, \sigma(\l^\circ, z)}, \quad \l, \l^\circ, z \in \R^{2d}
\end{equation}

We say that a lattice is symplectic if its generating matrix is a multiple of a symplectic matrix, i.e., $\L = c \, S \Z^{2d}$ with $c > 0$ and $S \in Sp(d)$, with $Sp(d)$ being the set of all symplectic $2d \times 2d$ matrices. We note that symplectic matrices actually form a group under matrix multiplication and that any symplectic matrix has determinant 1 and, hence, $Sp(d) \subset SL(2d,\R)$. In general $Sp(d)$ is a proper subgroup of the special linear group $SL(2d,\R)$, only for $d = 1$ we have that $Sp(1) = SL(2,\R)$. In particular, it follows that any 2--dimensional lattice is symplectic. In general, it follows from the definition of a symplectic matrix that
\begin{equation}
	\L^\circ = \vol(\L)^{-1/d} \L, \quad \L \text{ symplectic},
\end{equation}	
because, by definition, $S \in Sp(d) \Leftrightarrow S = J S^{-T} J^{-1}$ and for $\L = c \, S \Z^2$, $c > 0$ we have $\L^\circ = c^{-1} J S^{-T} J^{-1} \Z^{2d} = c^{-1} J S^{-T} \Z^{2d}$, as $J^{-1}$ is just another choice of basis for $\Z^{2d}$. Hence, for $\L$ symplectic the adjoint lattice is only a scaled version of the original lattice.

As a last point in this section, we have a closer look at the relation between the ambiguity function (and hence the STFT) and the Wigner distribution. We start with their relation given by the symplectic Fourier transform.
\begin{proposition}\label{pro_symp_Fourier}
	For $f,g \in \Lt$, the ambiguity function and the Wigner distribution are symplectic Fourier transforms of each other, i.e,
	\begin{equation}
		\F_\sigma \left(\A_g f\right) (x,\omega) = \W_g f(x, \omega) \quad \text{ and } \quad \F_\sigma \left(\W_g f\right) (x,\omega) = \A_g f(x, \omega)
	\end{equation}
\end{proposition}
Also, we have the following algebraic relation between the ambiguity function and the Wigner distribution.
\begin{proposition}\label{pro_algebraic}
	For $f,g \in \Lt$, the ambiguity function and the Wigner distribution fulfill
	\begin{equation}
		\W_g f(x,\omega) = 2^d \A_{g^\vee} f(2x, 2\omega) \qquad \text{ and } \qquad
		\A_g f(x, \omega) = 2^{-d} \W_{g^\vee} f \left( \tfrac{x}{2}, \tfrac{\omega}{2} \right),
	\end{equation}
	where $g^\vee (t) = g(-t)$ denotes the reflection of $g$.
\end{proposition}

We proceed with some more results regarding the ambiguity function and the Wigner distribution which we will need in the end to prove our main result. But first, we introduce some notation. For a function $F$ in phase space, the isotropic dilation is given by
\begin{equation}
	D_\alpha F(x, \omega) = F(\alpha x, \alpha \omega), \quad \alpha \in \R_+.
\end{equation}
The behavior of this operator under the symplectic Fourier transform is given by
\begin{equation}\label{eq_symplectic_FT_dilation}
	\F_\sigma (D_\alpha F)(x, \omega) = \alpha^{-2d} D_{\frac{1}{\alpha}} \, \F_\sigma \, F(x, \omega).
\end{equation}
\begin{lemma}\label{lem_Fourier_symp}
	For $f,g \in \Lt$ with $g^\vee = g$ we have
	\begin{equation}
		\begin{aligned}
			\F_\sigma \left(D_{\sqrt{2}} \, \A_g f \right)(x, \omega) & = D_{\sqrt{2}} \, \A_g f(x, \omega),\\
			\F_\sigma \left(D_{\frac{1}{\sqrt{2}}} \, \W_g f \right)(x, \omega) & = D_{\frac{1}{\sqrt{2}}} \, \W_g f(x, \omega).
		\end{aligned}
	\end{equation}
	If $-g^\vee = g$ we have
	\begin{equation}
		\begin{aligned}
			\F_\sigma \left(D_{\sqrt{2}} \, \A_g f \right)(x, \omega) & = - D_{\sqrt{2}} \, \A_g f(x, \omega),\\
			\F_\sigma \left(D_{\frac{1}{\sqrt{2}}} \, \W_g f \right)(x, \omega) & = - D_{\frac{1}{\sqrt{2}}} \, \W_g f(x, \omega).
		\end{aligned}
	\end{equation}
\end{lemma}
\begin{proof}
	This is an immediate consequence of Proposition \ref{pro_symp_Fourier} and Proposition \ref{pro_algebraic}. By using Proposition \ref{pro_symp_Fourier} and \eqref{eq_symplectic_FT_dilation} we get
	\begin{equation}
		\F_\sigma (D_{\sqrt{2}} \, \A_g f) (x, \omega) =  2^{-d} \, D_{\frac{1}{\sqrt{2}}} \W_g f (x, \omega)
		= D_{\sqrt{2}} \left(2^{-d} \W_g f \left( \tfrac{x}{2}, \tfrac{\omega}{2} \right) \right).
	\end{equation}
	Now, by the algebraic property from Proposition \ref{pro_algebraic} we conclude that
	\begin{equation}
		\F_\sigma (D_{\sqrt{2}} \, \A_g f)(x, \omega) = D_{\sqrt{2}} \, \A_{g^\vee} f(x, \omega).
	\end{equation}
	In a similar manner we derive the analogous statement for $\W_g f$. The results follow from the definitions of $\A_g f$ and $\W_g f$ and the assumptions that $\pm g^\vee = g$.
\end{proof}
In \cite{PeiLiu_FourierHermite_2012} it was shown that the (suitably scaled) cross Wigner distributions of two Hermite functions as well as tensor products of Hermite functions are eigenfunctions of the planar (2--dimensional) Fourier transform with eigenvalues $\pm 1$, depending on the pairing. In \cite{Lanzara_2017} another example of a ``nonstandard" eigenfunction of the planar Fourier transform was given, namely the function $F(x, \omega) = \frac{\sqrt{x^2+\omega^2}}{x \omega}$ (integrals have to be understood as Cauchy principal values in this case). All these examples are invariant under rotation (also the presented set of eigenfunctions is countable). Lemma \ref{lem_Fourier_symp} gives us an uncountable set of examples of eigenfunctions of the symplectic Fourier transform which do not necessarily possess any rotational symmetries.

For the next result, we recall that $\W_g f \in L^1(\R^{2d})$ if and only if $f,g \in S_0(\Rd)$ (see \cite[chap.~7]{Gosson_Wigner_2017} or \cite{Jakobsen_S0_2018}). Also, if $f,g \in S_0(\Rd)$, then the Wigner distribution $\W_g f$ is in $S_0(\R^{2d})$ (see \cite{Jakobsen_S0_2018}), which means that
\begin{equation}
	\W_g f \in L^1(\R^{2d}) \Longleftrightarrow \W_g f \in S_0(\R^{2d}).
\end{equation}
This statement holds, of course, for the ambiguity function $\A_g f$ and for the STFT $\V_g f$. Also, the assumptions for Poisson's summation formula to hold pointwise are met and we derive the following result.
\begin{lemma}\label{lem_sum_0}
	Let $f,g \in S_0(\Rd)$ and let $g$ be an odd function and $\L$ a symplectic lattice with $\vol(\L)^{-1} = 2^d$. Then
	\begin{equation}
		\begin{aligned}
			\sum_{\l \in \L} \W_g f(\l) & = - \sum_{\l \in \L} \W_g f(\l) = 0, \\
			\sum_{\l^\circ \in \L^\circ} \A_g f(\l^\circ) & = - \sum_{\l^\circ \in \L^\circ} \A_g f(\l^\circ) = 0.
		\end{aligned}
	\end{equation}
\end{lemma}
\begin{proof}
	By the symplectic version of Poisson's summation formula we have
	\begin{equation}
		\sum_{\l \in \L} \W_g f(\l) =  \underbrace{\vol(\L)^{-1}}_{=2^d} \sum_{\l^\circ \in \L^\circ} \F_\sigma \left( \W_g f \right)(\l^\circ)
	\end{equation}
	By Proposition \ref{pro_symp_Fourier} we have
	\begin{equation}
		2^d \sum_{\l^\circ \in \L^\circ} \F_\sigma \left( \W_g f \right)(\l^\circ) = 2^d \sum_{\l^\circ \in \L^\circ} \A_g f(\l^\circ)
	\end{equation}
	and by the algebraic relation in Proposition \ref{pro_algebraic} we have
	\begin{equation}
		2^d \sum_{\l^\circ \in \L^\circ} \A_g f(\l^\circ) = 2^d \sum_{\l^\circ \in \L^\circ} 2^{-d} \W_{g^\vee} f(2^{-1} \l^\circ) = - \sum_{\l \in \L} \W_g f(\l),
	\end{equation}
	since $g^\vee = - g$, $\vol(\L)^{-1} = 2^d$ and, hence, $2^{-1} \L^\circ = \L$ as $\L$ is symplectic. Therefore, the statement about the Wigner distribution follows. The statement for the ambiguity function follows analogously.
\end{proof}
An alternative (but equivalent) proof can be established by using Lemma \ref{lem_Fourier_symp}: Let $f, g \in S_0(\Rd)$, $-g^\vee = g$ and $\vol(\L)^{-1} = 1$ (note that in this case $\L = \L^\circ$), then
\begin{equation}
	\sum_{\l \in \L} D_{\frac{1}{\sqrt{2}}} \W_g f (\l) = \sum_{\l^\circ \in \L^\circ} \F_\sigma \left(D_{\frac{1}{\sqrt{2}}} \W_g f \right) (\l) = \sum_{\l \in \L} - D_{\frac{1}{\sqrt{2}}} \W_g f (\l).
\end{equation}
The analogous statement for $\A_g f$ obviously holds as well. Now, note that dilating the lattice and dilating the Wigner distribution are two equivalent ways to establish the result.

\section{Sharp Frame Bounds}\label{sec_sharp_bounds}
In this section we have a closer look at the frame operator and its spectrum. We will mainly follow Janssen's articles \cite{Jan95, Jan96}. The main differences are that we formulate the results for symplectic lattices in $2d$--dimensional phase space rather than for separable lattices in 2--dimensional phase space. Also, we use symmetric time--frequency shifts which only changes the appearing phase factors. For non--separable lattices, they will be easier to handle later on with this approach. Building on the results of the previous section, we will finally show that for odd windows in $S_0(\Rd)$ and $\L \subset \R^{2d}$ a lattice in phase space with $\vol(\L)^{-1} = 2^d$, the lower frame bound of the Gabor system $\widetilde{\G}(g,\L)$ vanishes. By the comments in Section \ref{subsec_symmetric} this is equivalent to the fact that the Gabor system $\G(g,\L)$ does not generate a frame, which is our main result.

The frame operator associated to the Gabor system $\widetilde{\G}(g, \L)$ is denoted by $\widetilde{S}_{g,\L}$ and given by
\begin{equation}
	\widetilde{S}_{g,\L} f = \sum_{\l \in \L} \left\langle f, \rho(\l) g \right\rangle \rho(\l) g, \qquad f \in \Lt.
\end{equation}
Another, very useful, representation of the frame operator is due to Janssen \cite{Jan95} and usually called Janssen's representation of the frame operator
\begin{equation}
	\widetilde{S}_{g,\L} = \vol(\L)^{-1} \sum_{\l^\circ \in \L^\circ} \left\langle g, \rho(\l^\circ) g \right\rangle \rho(\l^\circ).
\end{equation}
The frame operator is the composition of the analysis and the synthesis operator, which are adjoint to each other. The analysis operator maps a function from $\Lt$ to $\lp{2}(\L)$, $\L \subset \R^{2d}$ and is given by
\begin{equation}
	\widetilde{G}_{g,\L} f = \left( \langle f, \rho(\l) g \rangle \right)_{\l \in \L}.
\end{equation}
Its adjoint is called the synthesis operator, mapping sequences $c = (c_\l)_{\l \in \L} \in \lp{2}(\L)$ to $\Lt$, and is given by
\begin{equation}
	\widetilde{G}^*_{g,\L} c = \sum_{\l \in \L} c_\l \, \rho(\l) g.
\end{equation}
The frame operator can be written as
\begin{equation}
	\widetilde{S}_{g,\L} = \widetilde{G}^*_{g,\L} \widetilde{G}_{g,\L}.
\end{equation}
The following result is a straightforward generalization of the main result in \cite{Jan95}, where Janssen showed it for $d=1$ and $\L$ separable.
\begin{proposition}
	The following are equivalent:
	\begin{enumerate}[label=(\roman*)]
		\item $\widetilde{\G}(g, \L)$ is a frame with bounds $A$ and $B$.		
		\item $A \, I_{\Lt} \leq \widetilde{S}_{g,\L} \leq B \, I_{\Lt}$.
		\item $A \, I_{\lp{2}(\L^\circ)} \leq \vol(\L)^{-1} \widetilde{G}_{g,\L^\circ} \widetilde{G}^*_{g,\L^\circ} \leq B \, I_{\lp{2}(\L^\circ)}$.
	\end{enumerate}
\end{proposition}
The most interesting part for this work is that we can compute the frame bounds via the eigenvalues of the bi--infinite matrix, indexed by the adjoint lattice;
\begin{equation}
	\widetilde{G}_{g,\L^\circ} \widetilde{G}^*_{g,\L^\circ} = \left( \langle \rho(\l^\circ) g, \rho(\l^\circ{'}) g \rangle \right)_{\l^\circ, \l^\circ{'} \in \L^\circ}.
\end{equation}
We proceed by calculating the values of the above matrix;
\begin{equation}
	\langle \rho(\l^\circ) g, \rho(\l^\circ{'}) g \rangle = \langle g, \rho(-\l^\circ) \rho(\l^\circ{'}) g \rangle = e^{\pi i \, \sigma (\l^\circ, \l^\circ{'})} \langle g, \rho(\l^\circ{'}-\l^\circ) g \rangle, \qquad \l^\circ, \l^\circ{'} \in \L^\circ.
\end{equation}
Assume that $\vol(\L)^{-1/d} \in \N$, then the entries in $\vol(\L)^{-1} \widetilde{G}_{g,\L^\circ} \widetilde{G}^*_{g,\L^\circ}$ are constant along diagonals, i.e., $\vol(\L)^{-1} \widetilde{G}_{g,\L^\circ} \widetilde{G}^*_{g,\L^\circ}$ has a Laurent structure. For the time--frequency shifts $\rho(\l^\circ{'}-\l^\circ)$ the argument is obvious, the only justification we have to make is that the phase factor $e^{\pi i \, \sigma(\l^\circ,\l^\circ{'})}$ is constant along diagonals. We show that $\sigma(\l^\circ,\l^\circ{'})$ is an integer multiple of $\vol(\L)^{-1/d}$. Let $\L = \alpha ^{1/2d} S \Z^{2d}$, then $\vol(\L) = \alpha$ and $\L^\circ = \alpha^{-1/2d} S \Z^{2d}$. Since our lattice is symplectic by assumption, the symplectic form  $\sigma$ is independent from the matrix $S$ and we have
\begin{equation}
	e^{\pi i \, \sigma(\l^\circ,\l^\circ{'})} = e^{\pi i \, \sigma\left(\alpha^{-1/2d} S \, (k,l)^T, \, \alpha^{-1/2d} S \, (k',l')^T\right)} = e^{\vol(\L)^{-1/d} \pi i (k \cdot l' - k' \cdot l)}, \qquad k,l,k',l' \in \Z^d.
\end{equation}
In the case that $\vol(\L)^{-1/d}$ is even, the phase--factor equals $+1$ and can be neglected.

However, if $\vol(\L)^{-1/d}$ is odd, the phase--factor takes the role of an alternating sign, which is constant along diagonals, i.e., it is either $+1$ or $-1$ depending on the diagonal built by $\l^\circ - \l^\circ{'}$ being constant. For this reason we focus on the case where $\vol(\L)^{-1/d}$ is even.

It follows from the general theory on Toeplitz (matrices) and Laurent operators that the spectrum of such a (double) bi--infinite matrix can be computed via the essential infimum and supremum of a Fourier series, where the coefficients of the series are derived from the entries in the matrix. Using the above arguments, the following result is a straight forward generalization of  the result derived by Janssen in \cite{Jan96} (see Appendix \ref{app_Janssen} for Janssen's result).
\begin{proposition}\label{pro_series}
	For $g \in \Lt$ and $\L \subset \R^{2d}$ with $\vol(\L)^{-1/d} \in \N$ the Gabor system $\widetilde{\G}(g, \L)$ possesses the optimal frame bounds
	\begin{align}
		A & = \essinf_{z \in \R^2} \, \vol(\L)^{-1} \sum_{\l^\circ{'} - \l^\circ \in \L^\circ} e^{\pi i \, \sigma(\l^\circ, \l^\circ{'})} \A g(\l^\circ{'}-\l^\circ) \, e^{2 \pi i \, \sigma(\l^\circ{'}-\l^\circ, z)}\\
		B & = \esssup_{z \in \R^2} \, \vol(\L)^{-1} \sum_{\l^\circ{'} - \l^\circ \in \L^\circ} e^{\pi i \, \sigma(\l^\circ, \l^\circ{'})} \A g(\l^\circ{'}-\l^\circ) \, e^{2 \pi i \, \sigma(\l^\circ{'}-\l^\circ, z)}.
	\end{align}
\end{proposition}
The above series is real--valued, since we sum over a lattice the imaginary parts also appear as complex conjugates an cancel out. Note that the above series need not be convergent. In this case the upper bound might not be finite and the Gabor system might not constitute a frame. However, for windows in Feichtinger's algebra the upper bound is always finite\footnote{It follows from the results in Tolimieri and Orr \cite{TolOrr95} that $\vol(\L)^{-1} \sum_{\l^\circ \in \L^\circ} |\A g(\l^\circ)|$ always is an upper bound, however, usually not the optimal upper bound. For $g \in S_0(\Rd)$ this expression is always finite.}. As the frame operator $\widetilde{S}_{g,\L}$ is self--adjoint and positive semi--definite we have
\begin{equation}
	0 \leq A = \essinf_{z \in \R^2} \, \vol(\L)^{-1} \sum_{\l^\circ{'} -\l^\circ \in \L^\circ} e^{\pi i \, \sigma(\l^\circ, \l^\circ{'})} \A g(\l^\circ{'}-\l^\circ) \, e^{2 \pi i \, \sigma(\l^\circ{'}-\l^\circ, z)},
\end{equation}
by the theory of Laurent operators. We have now all the tools we need to prove Theorem \ref{thm_main}.

\subsection*{Proof of Theorem \ref{thm_main}}
In order to prove our main result, we will show that the lower frame bound vanishes under the assumptions of Theorem \ref{thm_main}. For $\vol(\L)^{-1} = 2^d$ and due to the fact that $\l^\circ{'} - \l^\circ \in \L^\circ$, the series in Proposition \ref{pro_series} reduces to
\begin{equation}
	\phi(z) = \vol(\L)^{-1} \sum_{\l^\circ \in \L^\circ} \A g(\l^\circ) \, e^{2 \pi i \, \sigma(\l^\circ, z)}.
\end{equation}
Now observe that
\begin{equation}
	\phi(0) = \vol(\L)^{-1} \sum_{\l^\circ \in \L^\circ} \A g(\l^\circ),
\end{equation}
which is, up to the factor $\vol(\L)^{-1}$, just the series from Lemma \ref{lem_sum_0}. Hence, we conclude that for $\vol(\L)^{-1} = 2^d$ and $g \in S_0(\Rd)$, $g^\vee = -g$ we have $\phi(0) = 0$. This is equivalent to the statement that the lower frame bound of the system $\widetilde{\G}(g, \L)$ vanishes. The same is true for the lower frame bound of the Gabor system $\G(g,\L)$. Hence, the proof of Theorem \ref{thm_main} is complete.

\begin{appendix}
\section{Janssen's Proposition}\label{app_Janssen}
	We will shortly state Janssen's proposition from \cite{Jan96} which he used to compute sharp frame bounds for the $\Lt[]$ case. In his formulation, Janssen used the STFT rather than the ambiguity function and separable lattices rather than general lattices. Hence, Janssen's result is a special case of Proposition \ref{pro_series}, but already carries the general idea in it. For $g \in \Lt[]$ and a lattice $\L_{(\alpha,\beta)} = \alpha \Z \times \beta \Z$, $(\alpha \beta)^{-1} \in \N$, the Gabor system $\G(g,\L_{(\alpha, \beta)})$ possesses the optimal frame bounds
	\begin{align}
		A & = \essinf_{(x,\omega) \in \R^2} \, (\alpha \beta) ^{-1} \sum_{k-k',l-l' \in \Z} \V g\left( \tfrac{k-k'}{\beta}, \tfrac{l-l'}{\alpha} \right) \, e^{2 \pi i \, ((k-k') x + (l-l') \omega)}\\
		B & = \esssup_{(x,\omega) \in \R^2} \, (\alpha \beta) ^{-1} \sum_{k-k',l-l' \in \Z} \V g\left( \tfrac{k-k'}{\beta}, \tfrac{l-l'}{\alpha} \right) \, e^{2 \pi i \, ((k-k') x + (l-l') \omega)}
	\end{align}
\end{appendix}
We note that the phase factor is now implicitly appearing in the STFT and that the standard Poisson summation formula (and not its symplectic version) was used.

%\bibliographystyle{plain}
%\bibliography{../../mybib}
\end{document}